\documentclass[11pt]{amsart}
\usepackage{amssymb}
\usepackage{psfrag}
\usepackage{wrapfig} 
\usepackage{amsfonts}
\usepackage{amsmath}
\usepackage{graphicx}
\newtheorem{theorem}{Theorem}[section]
\newtheorem{lemma}[theorem]{Lemma}
\newtheorem{prop}[theorem]{Proposition}
\newtheorem{cor}[theorem]{Corollary}

\newtheorem*{defin}{Definition}

\newtheorem*{PTC}{Parabolic-Type Test}
\theoremstyle{remark}
\newtheorem{example}[theorem]{Example}
\newtheorem*{remark}{Remark}

\makeatletter
\newcommand{\mysubsection}%
{\@startsection{subsection}{2}{\z@}{-3.25ex plus -1ex minus -.2ex}{-1ex}{\normalsize\sc}}
\makeatother

\renewcommand{\phi}{\varphi}
\renewcommand{\epsilon}{\varepsilon}
\newcommand{\U}{\mathbb{U}}

\newcommand{\C}{\mathbb{C}}

\newcommand{\dw}{\omega}

\newcommand{\htwo}{H^2}

\renewcommand{\notin}{\not\in}

\renewcommand{\Re}{{\rm Re\,}}

\newcommand{\Sp}{\text{Sp}}

\begin{document}



\title[Spectra of composition operators]{Spectra of some composition operators and associated weighted composition operators}
\author{Paul S. Bourdon}
\address{Department of Mathematics\\  Washington and Lee University, Lexington VA 24450}
\email{pbourdon@wlu.edu}

\begin{abstract}
{\scriptsize
We characterize the spectrum and essential spectrum of  ``essentially linear fractional'' composition operators acting on the Hardy space $H^2(\U)$ of the open unit disc $\U$. When the symbols of these composition operators have Denjoy-Wolff point on the unit circle, the spectrum and essential spectrum coincide.  Our work  permits us to describe the spectrum and essential spectrum of certain associated weighted composition operators on $H^2(\U)$.}
\end{abstract}

\maketitle
 
 \section{Introduction}

 Let $\U$ be the open unit disc in the complex plane, let $H(\U)$  be the space of analytic functions on $\U$, and let $H^2(\U)$ be the classical Hardy space, consisting of those functions in $H(\U)$ whose Maclaurin coefficients are square summable.  For $\phi$ an analytic selfmap of $\U$, let $C_\phi$ be the composition operator with symbol $\phi$ so that $C_\phi f = f\circ \phi$ for any $f \in H(\U)$.   Clearly $C_\phi$ preserves $H(\U)$.  Littlewood \cite{Lit} proved $C_\phi$ also preserves $H^2(\U)$; and thus, by the closed-graph theorem, $C_\phi:H^2(\U)\rightarrow H^2(\U)$ is a bounded linear operator.    
 
  Following \cite[p.\  48]{BLNS}, we say an analytic selfmap $\phi$ of $\U$ is {\it essentially linear fractional} provided 
  \begin{itemize}
  \item[(a)] $\phi(\U)$ is contained in a proper subdisc of $\U$ internally tangent to the unit circle at $\eta \in \partial \U$;
  \item[(b)] $\phi^{-1}(\{\eta\}):= \{\zeta \in \partial \U: \eta\ \text{belongs to the cluster set of}\ \phi\ \text{at}\ \zeta\}$  consists of one element, say $\zeta_0 \in \partial \U$; and 
 \item[(c)]  $\phi'''$ extends continuously to $\U\cup \{\zeta_0\}$.
 \end{itemize}
 For instance,  $\phi(z)=  \frac{2z^2-3z+3}{2z^2-7z+7}$ is essentially linear fractional with $\eta = \zeta = 1$ (see Example~\ref{ELFSQE} below).   Theorem 7.6 of \cite{BLNS} shows that for each essentially linear fractional selfmap $\phi$ of $\U$, there is a linear fractional selfmap $\psi$ of $\U$ (satisfying $\phi^{(j)}(\zeta) = \psi^{(j)}(\zeta)$ for $j \in \{0, 1, 2\}$)  such that $C_\phi - C_\psi$ is a compact operator.  Thus, an essentially linear fractional composition operator differs from a linear fractional composition operator by a compact operator.   However, this compact-difference condition does not characterize essentially linear fractional composition operators: there are, in fact, many linear fractional selfmappings of $\U$ that are not essentially linear fractional; for example, no automorphism of $\U$ is essentially linear fractional, and if $\psi$ is a linear fractional selfmapping of $\U$  such that $\psi(\partial \U)$ does not contact $\partial \U$, then $\psi$ is not essentially linear fractional.

 Much is known about spectra of composition operators on $H^2(\U)$; see, e.g., Chapter 7 of \cite{CMB}.  However, for many composition operators on $H^2(\U)$---including essentially linear fractional composition operators, complete spectral characterizations have not been obtained.  Moreover, there are interesting general questions that remain open, two of which we highlight below. 
 
   The analysis of the spectral behavior of $C_\phi:H^2(\U)\rightarrow H^2(\U)$ is typically case based, with the cases depending upon the {\it type} of the symbol $\phi$: ``dilation,'' ``hyperbolic,'' ``parabolic automorphism,'' and ``parabolic nonautomorphism.''  (We review the notion of type in the ``Preliminaries'' section below.)  Here are two of the questions addressed in this paper: 
 \begin{itemize}
 \item[Q1]   For $\phi$ of hyperbolic type or either of the parabolic types, do the spectrum and essential spectrum of $C_\phi$ always coincide?
 \item[Q2]    Let $r(C_\phi)$ denote the spectral radius of $C_\phi$ and $r_e(C_\phi)$, the essential spectral radius.   When $\phi$ is of dilation type and is either univalent or analytic on the closed disc, then work by Cowen and MacCluer  \cite{CMSP} and Kamowitz \cite{KM}  shows that the spectrum of $C_\phi$ consists of a disc $D$ (possibly degenerate) centered at the origin of radius $r_e(C_\phi)$ together with isolated eigenvalues.   In this situation, must every  point in $D$ be in the essential spectrum of $C_\phi$?  
 \end{itemize}
 In this paper, we characterize the spectrum and essential spectrum of  composition operators on $H^2(\U)$ induced by  selfmappings $\phi$ of $\U$ that are essentially linear fractional. For such  mappings $\phi$, our work shows that ``yes'' is the answer to Q1.  For question Q2, we show that the ``disc plus isolated eigenvalues'' characterization of the spectrum continues to be valid for  composition operators whose symbols are dilation-type essentially linear fractional  maps (which need not be univalent, or analytic on the closed disc), and we show that the answer to Q2 is also yes for such composition operators.    We also show that the spectrum and essential spectrum of certain weighted composition operators coincide.  
  
  This paper is organized as follows. In the next section, we set the stage for our work, providing needed background information.  In Section 3, we present some general results about spectra of composition operators; for example, we show that for any self-mapping $\phi$ of $\U$ of hyperbolic type or either of the parabolic types, if $\lambda$ is an eigenvalue of $C_\phi$ having an outer function as a corresponding eigenvector, then $\lambda$ does not belong to the compression spectrum of $C_\phi$; i.e., $C_\phi - \lambda I$ has dense range.   We also show that no eigenvalue of $C_\phi$ can be isolated for certain selfmaps $\phi$ of parabolic nonautomorphism type. Using results in Section 3 as lemmas, we characterize in Section 4 the spectrum and essential spectrum of $C_\phi$ when $\phi$ is essentially linear fractional.  We show, e.g., that if $\phi$ is an essentially linear fractional mapping of parabolic nonautomorphism type, then the spectrum and essential spectrum of $C_\phi$ coincide, each equaling a spiral $\{e^{-at}: t\ge 0\}$, where $a$ is the second derivative of $\phi$ at its Denjoy-Wolff point. (This characterization is consistent with one obtained by Cowen \cite[Corollary 6,2]{Cow2} for a special family of composition operators whose symbols are univalent and of parabolic nonautomorphism type.)   In Section 5, we characterize the spectrum and essential spectrum of certain weighted composition operators $f\mapsto gf\circ \phi$, where $g$ is bounded and analytic on $\U$ and $\phi$ is essentially linear fractional.  For example,  a consequence of Theorem~\ref{HCWCT} is that if both $g$ and $\phi$ are analytic on the closure of $\U$ and $\phi$ is an essentially linear fractional map of hyperbolic type, then the spectrum and essential spectrum of the weighted composition operator $C_{g,\phi}:  H^2(\U) \rightarrow H^2(\U)$ coincide, each equaling the disc $\{z: |z| \le |g(\dw)|\phi'(\dw)^{-1/2}\}$, where $\dw$ is the Denjoy-Wolff point of $\phi$.   (Here, $C_{g,\phi} f = g f\circ \phi$).   Our spectral characterizations in Section 5 appear to the be the first such characterizations for non power-compact weighted composition operators on $H^2(\U)$.  
\section{Preliminaries}

For detailed information about the Hardy space $H^2(\U)$ as well as inner and outer functions the reader may consult \cite{Dur}, for example.  Good general references for properties of selfmaps $\phi$ of $\U$ and of the composition  operators they induce on $H^2(\U)$  are \cite{CMB} and \cite{Sh2}.  We concentrate here on background information crucial to our work.

  \subsection{Reproducing kernels for $H^2(\U)$}   The Hardy space $H^2(\U)$ is a Hilbert space with inner product 
$$
\langle f, g\rangle = \sum_{n=0}^\infty \hat{f}(n)\overline{\hat{g}(n)},
$$
 where $(\hat{f}(n))$ and $(\hat{g}(n))$ are the sequences of Maclaurin coefficients for $f$ and $g$ respectively.  The norm of $f\in H^2(\U)$ is given by $\left(\sum_{n=0}^\infty |\hat{f}(n)|^2\right)^{1/2}$ or, alternatively, by
$$
\|f\|_{H^2(\U)}^2 = \frac{1}{2\pi} \int_0^{2\pi} |f(e^{it})|^2\, dt,
$$
where $f(e^{it})$ represents the radial limit of $f$ at $e^{it}$, which exists for a.e. $t\in [0, 2\pi)$ (with respect to Lebesgue measure).  Also, the inner product of two functions $f$ and $g$ in $H^2(\U)$ may be expressed as a boundary integral:
$$
\langle f, g\rangle = \frac{1}{2\pi} \int_0^{2\pi} f(e^{it})\overline{g(e^{it})}\, dt.
$$

Let $h\in H^{\infty}(\U)$, the Banach algebra of bounded analytic functions on $\U$ with $\|h\|_\infty = \sup\{|h(z)|: z\in \U\}$.   The integral representation of the $H^2(\U)$ norm makes it clear that  $\|hf\|_{H^2(\U)} \le \|h\|_\infty \|f\|_{H^2(\U)}$ for each $f\in H^2(\U)$.  Thus, the {\it multiplication operator} $M_h: H^2(\U) \rightarrow H^2(\U)$, defined by $M_h f = hf$, is bounded and linear on $H^2(\U)$.

For each $\alpha\in \U$, let $K_\alpha = 1/(1-\bar{\alpha}z)$.  Then $K_\alpha \in H^2(\U)$, and it is easy to see that  for each function $f\in H^2(\U)$, 
$$
\langle f, K_\alpha\rangle = f(\alpha);
$$
thus, $K_\alpha$ is the {\it reproducing kerne}l at $\alpha$ for $H^2(\U)$.  

	   It is easy to check that $(C_\phi)^*K_\alpha  = K_{\phi(\alpha)}$.  Also, the reproducing kernels $K_\alpha$, $\alpha \in \U$, are eigenfunctions for adjoint multiplication operators on $H^2(\U)$:  if $h\in H^\infty(\U)$ then $(M_h)^*K_\alpha = \overline{h(\alpha})K_\alpha$.  Finally, we have the following estimate for each $f\in H^2(\U)$ and each $z\in \U$:
	   \begin{equation}\label{PWB}
	   |f(z)| = |\langle f, K_z\rangle| \le \|f\|_{H^2(\U)}\|K_z\|_{H^2(\U)} = \frac{\|f\|_{H^2(\U)}}{\sqrt{1-|z|^2}}.
	   \end{equation}

\subsection{The Denjoy-Wolff Point $\dw$.}  Throughout this paper, $\phi$ will denote an analytic function on $\U$ for which $\phi(\U)\subseteq \U$; i.e., $\phi$ will always denote an analytic {\it  selfmap} of $\U$.  For $n$ a nonnegative integer, let $\phi^{[n]}$ denote the $n$-th iterate of $\phi$ so that, e.g.,  $\phi^{[0]}$ is the identity function on $\U$ and $\phi^{[2]} = \phi \circ \phi$.     If $\phi$ is not an elliptic automorphism of $\U$, then there is a  (unique)  point $\dw$ in the closure $\U^-$ of $\U$ such  that
$$
\dw = \lim_{n\rightarrow \infty} \phi^{[n]}(z)
$$
for each $z\in \U$.   The point $\dw$, called the {\it Denjoy-Wolff point} of $\phi$, is also
characterized as follows: if $|\dw| < 1$, then $\phi(\dw) = \dw$ and
$|\phi'(\dw)| < 1$; if
$\dw\in \partial \U$, then $\phi(\dw) =\dw$ and $0 < \phi'(\dw) \le 1$.  If
$|\dw| = 1$, then $\phi(\dw)$ represents the angular (non-tangential) limit of $\phi$
at
$\dw$ and $\phi'(\dw)$ represents the angular derivative of $\phi$ at $\dw$.    The location of the Denjoy-Wolff point and the behavior of iterate sequences $(\phi^{[n]}(z))$ as they approach the Denjoy-Wolff point strongly influence properties of the operator $C_\phi$.  For example, no $\phi$ with Denjoy-Wolff point on $\partial U$ can be the symbol of a compact composition operator on $H^2(\U)$ (see e.g., \cite[p.\ 56]{Sh2}).  The Denjoy-Wolff point plays a major role in the classification system for selfmaps of $\U$ presented in the next subsection.

\subsection{Type for Selfmaps of $\U$}    Let $\phi$ be a selfmap of $\U$ with Denjoy-Wolff point $\dw$.  We classify $\phi$ as follows (cf. \cite[Definition 0.3]{BoS}):
\begin{itemize}
\item  Dilation type:  $\dw \in \U$;
 \item Hyperbolic type: $\dw \in \partial \U$ and $\phi'(\dw)< 1$;
 \item Parabolic automorphism type: $\dw \in \partial \U$, $\phi'(\dw) =1$, and the iterate sequence $(\phi^{[0]}(0))$ is separated in the hyperbolic metric on $\U$;
 \item Parabolic nonautomorphism type:   $\dw \in \partial \U$, $\phi'(\dw) =1$, and the iterate sequence $(\phi^{[0]}(0))$ is not separated in the hyperbolic metric on $\U$.
 \end{itemize}

Assuming $\phi'(\dw) = 1$, distinguishing whether $\phi$ is of parabolic automorphism or nonautomorphism type can be difficult.  However, there is an easy test if $\phi$ has enough smoothness near its Denjoy-Wolff point.   

\begin{defin} Let $n$ be a positive integer, let $\zeta\in \partial \U$, and let $0 \le \epsilon < 1$. Following
\cite[p.~50]{BoS}, we say that the self-map
$\phi$ of\/
$\U$ belongs to $C^{n+\epsilon}(\zeta)$ provided that $\phi$ is
differentiable at $\zeta$ up to order $n$ $($viewed as a function
with domain $\U\cup \{\zeta\})$ and, for $z\in \U$, has the expansion
$$
\phi(z) = \sum_{k=0}^{n}\frac{\phi^{(k)}(\zeta)}{k!}(z-\zeta)^k
+ \gamma(z),
$$ 
where $\gamma(z) = o(|z-\zeta|^{n+\epsilon})$ as $z\rightarrow
\zeta$ from within $\U$.
\end{defin}

It is not difficult to show that $\phi\in C^{n}(\zeta)$ whenever $\phi^{(n)}$ extends
continuously to $\U\cup\{\zeta\}$.   
 
The following is Theorem 4.4  of \cite{BoS}.
\begin{PTC}  Suppose that $\phi\in C^2(\dw)$  and $\phi'(\dw) = 1$.  Then $\Re(\dw \phi''(\dw)) \ge 0$; moreover,
\begin{itemize}
\item[(a)]  if $\phi''(\dw) = 0$ or if $\Re(\dw \phi''(\dw)) > 0$, then $\phi$ is of parabolic nonautomorphism type;
\item[(b)] if $\dw\phi''(\dw)$ is pure imaginary (and nonzero) and $\phi\in C^{3+\epsilon}(\dw)$ for some positive $\epsilon$, then $\phi$ is of parabolic automorphism type.
\end{itemize}
\end{PTC}

\subsection{Horocyclic Selfmaps of $\U$}
 For $\zeta\in \partial \U$ and $\beta > 0$, let $H(\eta, \beta) = \{z: |1- z\bar{\eta}|^2 < \beta(1- |z|^2)\}$ be the open horodisc (of radius $\beta/(1+\beta)$) internally tangent to the unit circle at $\eta$. Call a self-map $\phi$ of $\U$  {\it horocyclic} at $\eta$  provided $\phi(\U)$ lies in a horodisc $H(\eta, \beta)$, for some $\beta >0$.    Note any essentially linear fractional selfmap of $\U$ must be horocyclic at $\eta$ for some $\eta\in \partial \U$.  
 
 In the next subsection, we present a criterion (Proposition~\ref{ITSP}) for $\phi$ to be horocyclic at $1$.   We apply the criterion to show $\phi(z) = 2/(\sqrt{13-4z}-1)$ is horocyclic at $1$  (see Example~\ref{EIHT}).

\subsection{Essentially Linear Fractional Selfmaps of $\U$} \label{ELFS} 
 Let $\phi$ be an essentially linear-fractional selfmapping of $\U$.  There are unimodular constants $\eta$ and $\zeta$ such that  $\phi\in C^3(\zeta)$, $\phi(\zeta) = \eta$, $\phi^{-1}(\{\eta\}) = \{\zeta\}$,  and $\phi$ is horocyclic at $\eta$.  Recall from the Introduction that in this context $\phi^{-1}(\{\eta\}) = \{\zeta\}$ means that $\zeta$ is the only point in $\partial \U$ whose cluster set under $\phi$ includes $\eta$.   It follows that if $\eta\ne \zeta$, then $\|\phi\circ\phi\|_\infty < 1$ (for otherwise there would be a sequence $(z_n)$ in $\phi(\U)$ such that $\lim_n|\phi(z_n)| = 1$, but since $\phi(U)$ is contained in a proper subdisc of $\U$ internally tangent to $\partial \U$ at $\eta$, then both $(z_n)$ and $(\phi(z_n))$ approach $\eta$,  making $\phi^{-1}(\{\eta\}) = \{\zeta, \eta\}$, a contradiction.)   Hence if $\eta\ne\zeta$, then $(C_\phi)^2$ is compact (see, e.g., \cite[p.\ 23]{Sh2}).  When some power of $C_\phi$ is compact the essential spectrum of $C_\phi$ is just $\{0\}$ and  a result of Caughran and Schwartz \cite[Theorem 3]{CSw} shows that the spectrum of $C_\phi$ is $\{0\}\cup \{\phi'(\dw)^n, n=0, 1, 2, \ldots\}$, where $\dw$ (necessarily in $\U$) is the Denjoy-Wolff point of $\phi$.  {\em Thus, in our discussions below concerning spectra  for essentially linear fractional composition operators, we will restrict our attention to the case where the essentially linear fractional symbol $\phi$ fixes a point $\zeta\in \partial \U$.    Without loss of generality, we assume $\zeta = 1$.}

Let $\phi$ be an {\em arbitrary} selfmap of $\U$ such that $\phi(1) = 1$ and $\phi \in C^2(1)$. Let $p= \phi'(1)$, $a = \phi''(1)$,  and let $T(z) = (1+z)/(1-z)$ so that $T$ maps $\U$ univalently onto the right halfplane $\Pi$.   Because $\phi\in C^2(1)$, the selfmapping $\Phi:=T\circ \phi \circ T^{-1}$ of $\Pi$ has the following representation:
\begin{equation}\label{RHPR}
 \Phi(w) = \frac{1}{p} w + \left(\frac{1}{p} - 1 + \frac{a}{p^2}\right) + \Gamma(w), w\in \Pi,
\end{equation}
 where $\Gamma(w)  = o(1)$ as $|w|\rightarrow \infty$ (cf. \cite[Equation (27)]{BLNS}).  Clearly, $\phi$ is horocyclic at $1$ if and only if  there is a constant $c>0$ such that
\begin{equation}\label{BBC}
\Re(\Phi(w)) \ge c \quad  \text{for all}\ w\in \Pi.
\end{equation}  
The representation (\ref{RHPR}) reveals that if (\ref{BBC}) holds, then $\Re\left(\frac{1}{p} - 1 + \frac{a}{p^2}\right) \ge c > 0$.   Thus  if $\phi$ is horocyclic at $1$, then $1/\phi'(1) - 1 + \phi''(1)/\phi'(1)^2$ must be positive.  The converse is true if, say, $\phi$ extends to be continuous on the closed disc and $|\phi(\zeta)| < 1$ for $\zeta\in \partial \U \setminus \{1\}$.  
  
 \begin{prop}\label{ITSP}  Let $\phi$ be an analytic selfmap of $\U$ that extends to be continuous on $\U^-$.  Suppose that $\phi\in C^2(1)$, that $\phi(1)=1$ and that $|\phi(\zeta)|<1$ for $\zeta\in \partial \U\setminus \{1\}$.  If 
 \begin{equation}\label{LBC}
\Re\left(\frac{1}{\phi'(1)} - 1 + \frac{\phi''(1)}{\phi'(1)^2}\right)> 0
\end{equation}
then $\phi(\U)$ is contained in a proper subdisc of $\U$ internally tangent to $\partial \U$ at $1$; i.e., $\phi$ is horocyclic at $1$.
\end{prop}
\begin{proof}   Suppose  that (\ref{LBC}) holds.    Set $\Phi:=T\circ\phi \circ T^{-1}$.  We know that the proposition follows if the real part of $\Phi$ is bounded below by a positive number.  Suppose, in order to obtain a contraction that $\Re(\Phi)$ is not bounded below, so that there is a sequence $(w_n)$ in the  right halfplane for which $\Re(\Phi(w_n)) \rightarrow 0$.   

Since $\Phi$ has the representation (\ref{RHPR}) and (\ref{LBC}) holds, we see that $(w_n)$ must be a bounded sequence.  Thus $(w_n)$  has a limit point $b$ in the closure of the right halfplane.  Because $\phi$ is continuous on $\U^-$, $\Phi$ is continuous at $b$ and $\Re(\Phi(b)) = 0$.   Note that $\Re(\Phi(b)) = 0$ tells us that $b$ must be on the imaginary axis.   Hence, $\phi(T^{-1}(b)) = T^{-1}(\Phi(b))$ belongs to $\partial \U\setminus \{1\}$, contradicting the hypothesis that $|\phi(\zeta)|<1$ for $\zeta\in \partial \U\setminus \{1\}$. 
\end{proof}

Let $\phi$ be an essentially linear fractional selfmap of $\U$ with $\phi(1) = 1$ (so that $\phi$ is horocyclic at $1$).  Let $p= \phi'(1)$ and  $a = \phi''(1)$.   Because $\phi\in C^3(1)\subseteq C^2(1)$, the selfmapping $\Phi:=T\circ \phi \circ T^{-1}$ of $\Pi$ has the representation (\ref{RHPR}).    As we discussed above, because $\phi$ is horocyclic at $1$, (\ref{RHPR}) shows that there is a positive constant $c$ such that $\Re\left(\frac{1}{p} - 1 + \frac{a}{p^2}\right) \ge c > 0$.   Note in particular that if $p = 1$ ($\phi$ is parabolic type) then we must have $\Re(a) > 0$, so that $\phi$ is of nonautomorphism type by the Parabolic Type Test.    Thus, of the two types of parabolic selfmappings of $\U$, only those of nonautomorphism type can be essentially linear fractional.  Because $\Re\left(\frac{1}{p} - 1 + \frac{a}{p^2}\right) \ge 0$, 
 $$
 \Psi(w):= w/p + \left( \frac{1}{p} - 1 + \frac{a}{p^2}\right)
 $$
  is a selfmap of $\Pi$.  Theorem 7.6 of \cite{BLNS} shows that if $\psi$ is the linear-fractional map of $\U$ defined by $T^{-1}\circ \Psi \circ T$, then $C_\phi- \C_\psi$ is a compact operator on $H^2(\U)$.

  We conclude with three concrete examples of essentially linear fractional selfmappings of $\U$.

\begin{example}\label{EIHT} Let $\phi(z) = 2/(\sqrt{13-4z}-1)$  (where $\sqrt{\cdot}$ represents the principal branch of the square root function---all roots in this paper should be viewed as principal ones).     The function $z\mapsto 13-4z$ takes $\U$ to the disc centered at 13 of radius 4 and thus $z\mapsto |\sqrt{13-4z}-1|$ attains its absolute minimum value 2 on $\U^-$ at $z=1$ and only at $z=1$.  Thus $\phi$ is a selfmap of $\U$ such that $\phi(1) = 1$, $\phi^{-1}(\{1\}) = \{1\}$, and $\phi\in C^3(1)$ (in fact $\phi$ is analytic on $\U^-$).    One may apply Proposition~\ref{ITSP} to see that $\phi(\U)$ belongs to  proper subdisc of $\U$ internally tangent to $\U$ at $1$.  Thus $\phi$ is an essentially linear fractional selfmap of $\U$.   Since $\phi(1) = 1$ and $\phi'(1) = 1/3$, $\phi$ is of hyperbolic type.
\end{example}

\begin{example}\label{ELFSQE}  Constructing examples of essentially linear fractional selfmaps $\phi$ is quite simple using the representations like (\ref{RHPR}).  For instance, choosing $p = 1/2$, $a=1$ and $\Gamma(w) = 2/(w+1)$, we obtain $\phi(z) = T^{-1}(\Phi(T(z))) = \frac{2z^2-3z+3}{2z^2-7z+7}$ as an essentially linear fractional selfmap of $\U$ such that $\phi(1) = 1$ and $\phi'(1) = 2$.  That $\phi$ is horocyclic at $1$ is clear  from the form of $\Phi$, which shows $\Re(\Phi(w)) \ge 1$ for every $w\in \Pi$.  The form of  $\Phi$ also makes it clear that $\phi^{-1}(\{1\}) = \{1\}$.    Finally, note that $\phi$ must be of dilation type since $\phi'(1) > 1$ (in fact, $\phi(1/2) = 1/2$.)  We note that essentially linear fractional mappings need not be univalent on $\U$; in fact, this is the case for the the rational mapping $\phi$ we just constructed (its derivative vanishes at $0$).  
\end{example}

\begin{example}\label{EIPC}  For $w\in \Pi$, let 
$$
\Phi(w) = w + 2 +i  - \frac{1}{4(w+1)} - \frac{1}{4(w+1)^{3/2}}.
$$
  Note that $\Phi$ is a self-map of $\Pi$ and that, in fact, $\Re(\Phi(w)) > 3/2$ for every $w\in \Pi$ (since $|1/(4(w+1)) + 1/(4(w+1)^{3/2})| <1/2$ for $w\in \Pi$).  It follows that  
$$
\phi(z) = T^{-1}(\Phi(T(z))) = \frac{z^2 - (2-8i)z-(15+8i) + 2^{-1/2}(1-z)^{5/2}}{z^2+(14+8i)z-(31+8i) + 2^{-1/2}(1-z)^{5/2}}
$$
is an analytic selfmap of $\U$ that is horocyclic at $1$ (in fact, since $\Re(\Phi(w)) > 3/2$, the set $\phi(\U)$ must be contained in $T^{-1}(\{w: \Re(w) > 3/2\})$, which is the horodisc  $H(1,2/3)$ (having radius $2/5$ and center $3/5$)---see the figure below).   The form of  $\Phi$ ensures that $\phi$ is $C^{3+\epsilon}(1)$ for  $0 \le \epsilon < 1/2$ (see \cite[pp. 50--51]{BoS}) and $\phi^{-1}(\{1\}) = \{1\}$.  Since $\phi'(1) = 1$ (and $\phi''(1) = 2+i$ has positive real part) we see that $\phi$ is an essentially linear fractional selfmap of $\U$ of parabolic nonautomorphism type.  
\end{example}
\begin{figure}[h]
\psfrag{H}{{\small $\partial H(1,2/3)$}}
\psfrag{F}{{$\phi(\U)$}}
  \includegraphics[height=2in]{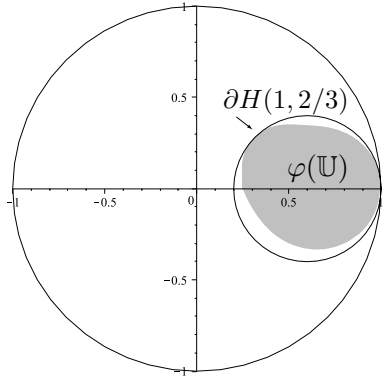}
 \caption{ The image of the unit disk under the essentially linear fractional mapping $\phi$ of Example~\ref{EIPC}} 
  \end{figure} 
  
  \subsection{Adjoints of Composition operators}\label{ACOS}  Recall that for $h\in H^{\infty}(\U)$,  $M_h$ denotes the operator of multiplication by $h$ on  $H^2(\U)$.  We let $B = (M_z)^*$ so that $B$ is the backward shift operator on $H^2(\U)$.  
  
     Let $\phi(z) = (az + b)/(cz+d)$ be a linear-fractional selfmap of $\U$.  Cowen \cite{Cow} derives the following formula for the adjoint of $C_\phi$:

\begin{equation}\label{CAF}
  C_\phi^* = M_gC_\sigma M_h^* 
\end{equation} 
  where $g(z) = 1/(-\bar{b}z + \bar{d})$, $h(z) = cz + d$, and $\sigma(z)= (\bar{a}z - \bar{c})/(-\bar{b}z +\bar{d})$.    Analogs of this adjoint formula have been developed for rational selfmappings of $\U$ (see, e.g.,  \cite{CG}, \cite{HMR}, and \cite{BSA}).     We need the following variant of (\ref{CAF}) (see \cite[Equation (2)]{HMR} or \cite[Equation (18)]{BSA}):
  \begin{equation}\label{ORFM}
  C_\phi^* =      \Lambda_0 + M_{z\sigma'} \, C_{\sigma} \, B,
\end{equation}
  where $\Lambda_0$ is the rank-one operator defined by $(\Lambda_0 f ) = f(0)/(1-\overline{\phi(0)}z)$.

  \subsection{Spectrum and Essential Spectrum}\label{SESS}   For $T$ an operator on a Hilbert space $H$, let $\Sp(T)$ denote the spectrum of $T$ and $\Sp_e(T)$ denote the essential spectrum.   We will make use of the following well-known fact from Fredholm theory:  any point $\lambda$ belonging to both the spectrum of $T$ and unbounded component of the complement of the essential spectrum must be an eigenvalue of $T$.    The argument is simple.    Because the Fredholm index is continuous on the complement of the essential spectrum and has value zero on the complement of the spectrum,  the index is zero on the unbounded component of the complement of the essential spectrum.  Thus, whenever $\lambda$ belongs to both $\Sp(T)$ and the unbounded component of the complement of $\Sp_e(T)$, then the Fredholm index of $T -\lambda I$ is zero.  Thus for such a $\lambda$, the operator $T - \lambda I$ has closed range and has kernel and co-kernel of the same dimension; it follows that the kernel must have positive dimension, for otherwise $C_\phi - \lambda I$ would be invertible.

 \section{Some General Results on Composition-Operator Spectra}\label{GRCOS}
 
 We begin with our most general result.

  \begin{prop}\label{OFP}   Let $\phi$ be a selfmap of $\U$ having Denjoy-Wolff point $\dw$ in $\partial \U$.  Suppose that  $\lambda$ is an eigenvalue of $C_\phi:H^2\rightarrow H^2$ having a corresponding eigenvector that is an outer function.   Then $C_\phi - \lambda I$ has dense range.
 \end{prop}
 
 \begin{proof}  Let $\lambda\in \C$ be an eigenvalue of $C_\phi$ with eigenvector $g$ being an outer function.  We show that $\overline{\lambda}$ is not an eigenvalue of $C_\phi^*$ to complete the proof of the proposition.
 
 Suppose, in order to obtain a contradiction, that $\overline{\lambda}$ is an eigenvalue of $C_\phi^*$ with corresponding eigenvector $h$.  Observe that  for any nonnegative integer $n$ and any positive integer $k$, we have
\begin{equation}\label{HCED}
\begin{split}
\lambda^k\langle z^n(\dw-z)g(z), h\rangle & =    \langle z^n(\dw-z)g(z), (C_{\phi}^*)^k h\rangle\\
    & =    \lambda^k \langle (\phi^{[k]})^n(\dw-\phi^{[k]})g,  h\rangle.
    \end{split}
\end{equation}
By Theorem 3.1 of \cite{BMS}, the iterate sequence  $(\phi^{[k]})$ converges to $\dw$ in the norm of $H^2(\U)$ and it follows that some subsequence $(\phi^{[j_k]})$ converges a.e.\ on $\partial \U$ to $\dw$. Because $|(\phi^{[j_k]})^n(\dw-\phi^{[j_k]})g\bar{h}| \le 2|g\bar{h}|\in L^1(\partial \U)$ and $(\phi^{[j_k]})$ converges to $\dw$ a.e. on $\partial \U$,  Lebesgue's Dominated Convergence Theorem shows that 
$\lim_{k\rightarrow \infty} \langle (\phi^{[j_k]})^n(\dw-\phi^{[j_k]})g,  h\rangle = 0$.  Thus, by (\ref{HCED}), we conclude that
$$
\langle z^n(\dw-z)g(z), h\rangle  = 0
$$
for all $n\ge 0$.  Since $z\mapsto(\dw-z)g(z)$ is outer (being a product of outer functions), the preceding equation shows that $h$ is orthogonal to the polynomial multiples of an outer function, a set which is a dense set in $H^2(\U)$ by Beurling's Theorem. Thus $h \equiv 0$, a  contradiction that completes the proof of the proposition.\end{proof}

We present two applications of the preceding proposition, the first stated as a corollary.

\begin{cor}  Suppose that $\phi$ is a self-mapping of $\U$ of hyperbolic type with Denjoy-Wolff point $\dw$.  Then $(C_\phi - \lambda I)$ has dense range whenever 
$\lambda$ belongs to the annulus $A:= \{z: \phi'(\dw)^{1/2} < |z| < \phi'(\dw)^{-1/2}\}$.  
\end{cor}

\begin{proof} Let $b$ be such that  $-1/2 < \Re(b) < 1/2$.  Then it is easy to see that $f_b(z)  := \left(\frac{1+z}{1-z}\right)^b$ belongs to $H^2(\U)$. Moreover, $f_b$ is outer (because, e.g., both $f_b$ and $1/f_b$ are in $H^2(\U)$).   Cowen \cite[Theorem 4.5]{Cow2} proves that there is a selfmap $\sigma$ of $\U$ such that $f_b\circ \sigma$ is an eigenvector for $C_\phi$ with corresponding eigenvalue $\phi'(\dw)^b$.  Because $b$ is an arbitrary complex number with $\Re(b)\in (-1/2,1/2)$, it follows that every point in the annulus $A$ is an eigenvalue of $C_\phi$ with corresponding eigenfunction having the form $f_b\circ \sigma$ for an appropriate $b$.   Since every composition operator preserves the collection of outer functions (see, e.g.,  \cite[Section 2.7]{CKS}), $f_b\circ \sigma$ is outer and the corollary follows.\end{proof}

Our second application of Proposition~\ref{OFP} concerns a family of composition operators introduced by Cowen \cite[Section 6]{Cow2}.  
Let  $T_{\theta}(z) = \left(\frac{1+z}{1-z}\right)^{2\theta/\pi}$ so that for $ 0 < \theta \le \pi/2$, $T_{\theta}$ maps $\U$ univalently into the right halfplane $\Pi$ and $T_{\pi/2} = T$ maps $\U$ onto $\Pi$.   For $a\ne 0$  such that $|\arg(a)| < \theta\le \pi/2$, define the selfmap $\phi_{a, \theta}$ of $\U$ by 
$$
\phi_{a,\theta} (z) = (T_{\theta})^{-1}\circ (T_{\theta}+a).
$$
\begin{prop} The compression spectrum of $C_{\phi_{a, \theta}}$ is empty, i.e., $(C_{\phi_{a,\theta}} - \lambda I)$ has dense range for every $\lambda$.  
\end{prop}
\begin{proof}   Let $\theta\in (0, \pi/2]$ and let $a\ne 0$ satisfy $|\arg(a)| < \theta$.  Note that  $\phi_{a, \theta}(\U)$ is a simply connected domain whose boundary is a Jordan curve. Thus, by Walsh's Theorem (see, e.g,, p.\ 11 of \cite{BoS}), the set of polynomials in $\phi_{a, \theta}$ is dense $H^2(\U)$; equivalently, $C_{\phi_{a, \theta}}$ has dense range.     

Cowen proves \cite[Corollary 6.2]{Cow} that 
$$
\Sp(C_{\phi_{a, \theta}}) = \{e^{-\beta a}: |\arg \beta| \le \pi/2- \theta\} \cup \{0\}.
$$
In fact, he shows that every nonzero spectral point is an eigenvalue of $C_{\phi_{a,\theta}}$.  We now examine the eigenfunctions that Cowen exhibits.  Let $\beta\ne 0$ belong to the sector $\{z: \arg(z) \le \pi/2-\theta\}$ so that $e^{-\beta a}$ belongs to the spectrum of $C_{\phi_{a, \theta}}$.   Note that  $\Re(\beta T_\theta(z)) > 0$ for every $z\in \U$ so that $F : = e^{-\beta T_\theta}\in H^\infty(\U)$.  Clearly $F$ is an eigenfunction for $C_{\phi_{a,\theta}}$ with corresponding eigenvalue $e^{-\beta a}$.   We complete the  argument by considering two cases.

Case 1: $\theta < \pi/2$.      We need to show  $C_{\phi_{a,\theta}} - e^{-\beta a}I$ has dense range.  This follows from Proposition~\ref{OFP}, because the eigenfunction $F=    e^{-\beta T_\theta}$ for $e^{-\beta a}$  is outer by the following simple computation:   
\begin{eqnarray*}
\frac{1}{2\pi} \int_{-\pi}^{\pi} \log|F(e^{it})|\, dt & = & \Re\left(\frac{-\beta}{2\pi} \int_{-\pi}^{\pi} T_\theta(e^{it}) \, dt\right)\\
                     						& = & 	\Re\left(-\beta\, T_\theta(0) \right)\\
						& = &\Re(-\beta)\\
						& =& \log|F(0)|,
						\end{eqnarray*}
						where the second equality follows, e.g., from Theorem 3.6 of \cite{Dur}, because $T(z) =\left(\frac{1+z}{1-z}\right)^{2\theta/\pi}$ is an $H^1(\U)$ function (since $2\theta/\pi< 1$).  

Case 2: $\theta = \pi/2$.  In this case, $\beta > 0$ and the eigenfunction $F(z) = e^{-\beta T_\theta(z)} = e^{-\beta (1+z)/(1-z)}$ is a singular inner function; in fact, the unit singular $S(z) = e^{-(1+z)/(1-z)}$ raised to the $\beta$ power.   Let $t$ be any nonnegative real number.  We have $(C_{\phi_{a, \theta}} -  e^{-\beta a}I) S^t= (e^{-ta} - e^{-\beta a})S^t$.  It follows that the closure of the range of $(C_{\phi_{a, \theta}} -  e^{-\beta a}I)$ contains all nonnegative powers of the unit singular function.  Since the linear span of the collection of all nonnegative powers of $S$ is dense in $H^2(\U)$ (see, e.g,  \cite[Lemma 4.2]{GM}),  $(C_{\phi_{a, \theta}} -  e^{-\beta a}I)$ has dense range, as desired.
\end{proof}

\begin{remark} The issue of when $(C_\phi - \lambda I)$ has dense range relates to Question Q1 raised in the Introduction as follows.  Suppose that $\phi$ is a selfmap of $\U$ or hyperbolic or parabolic automorphism type.  Then by      Corollary 4.4 and the remarks following Theorem 4.10 of \cite{Cow2}, every eigenvalue of $C_\phi$ has infinite multiplicity.  Thus if there is a number $\lambda\in \Sp(C_\phi)\setminus \Sp_e(C_\phi)$, then $C_\phi - \lambda I$  must be injective with closed range and it follows that the range of $(C_\phi - \lambda I)$ cannot be dense.  Thus for selfmaps $\phi$ of hyperbolic or parabolic automorphism type, Q1 has an affirmative answer if one can show that for every nonzero $\lambda$, the operator $C_\phi - \lambda I$ has dense range. 
\end{remark}
 
 The remainder of the results in this section serve as lemmas in the next section,  permitting us to characterize the spectrum and essential spectrum of essentially linear fractional composition operators (fixing a point on $\partial \U$).
 
 For $\eta\in \partial \U$ and $\beta > 0$, recall that  $H(\eta, \beta) = \{z: |1- z\bar{\eta}|^2 < \beta(1- |z|^2)\}$ is the open horodisc (of radius $\beta/(1+\beta)$) internally tangent to the unit circle at $\eta$. 
 
 \begin{prop} \label{HCP} Suppose that $\phi$ is a self-map of hyperbolic type that is horocyclic at its Denjoy-Wolff point $\dw \in\partial U$; then $\Sp_e(C_\phi) = \Sp(C_\phi)$.
 \end{prop}
 \begin{proof}  Without loss of generality, we take $\phi$'s Denjoy-Wolff point to be $1$.  Let $\beta > 0$ be such that $\phi(U)\subseteq H(1, \beta)$.  Then by Julia's Theorem \cite[p.\ 63]{Sh2},  for every $z \in \U$ and positive integer $n$, 
\begin{equation}\label{JTR}
 \left|1-\phi^{[n+1]}(z)\right|^2 < \phi'(1)^n\beta(1-|\phi^{[n+1]}(z)|^2) \le \phi'(1)^n\beta.
\end{equation}
Thus  $\|1-\phi^{[n+1]}\|_\infty \le \phi'(1)^n\beta$.

  Suppose that there is a number $\lambda\in \Sp(C_\phi)$ but not in $\Sp_e(C_\phi)$.   Since the only Fredholm composition operators are the invertible ones induced by automorphisms of $\U$ (\cite{CTW}), we know $\lambda \ne 0$.     Since $\phi$ is  of hyperbolic type, every eigenvalue of $C_\phi$ has infinite multiplicity (\cite[Corollary 4.4]{Cow2}).  Thus $\lambda$  cannot be an eigenvalue of $C_\phi$; i.e., $C_\phi - \lambda I$ is injective.    We arrive at a contradiction by showing that $\bar{\lambda}$ also cannot be an eigenvalue of $C_\phi^*$, which shows that $C_\phi - \lambda I$ has dense range and hence is surjective since its range must be closed.   
  
  Suppose,  in order to obtain a contradiction, that  $\bar{\lambda}$ is an eigenvalue of $C_\phi^*$ with corresponding eigenvector $g$.  Because $\lambda \ne 0$, we may choose the positive integer $k$ such that $\phi'(1)^k < |\lambda|$.  Because $z\mapsto (1-z)^k$ is an outer function and $g\in H^2(\U)$ is nonzero function, there is a nonnegative integer $m$ such that $\langle z^m(1-z)^k, g\rangle \ne 0$.  For each nonnegative integer $n$, we have
  \begin{eqnarray*}
  |\lambda|^{n+1} |\langle z^m(1-z)^k, g\rangle| & = &   |\langle z^m(1-z)^k, (C_\phi^*)^{n+1} g\rangle|\\
                                                                          & = &| \langle  (\phi^{[n+1]})^m(1-\phi^{[n+1]})^k, g\rangle|\\
                                                                              & \le & \| (\phi^{[n+1]})^m(1-\phi^{[n+1]})^k\|_\infty  \|g\|_{H^2(\U)}\\
                                                                             & \le & \phi'(1)^{nk}\beta^k\|g\|_{H^2(\U)}.
                                                                        \end{eqnarray*}                                                                                                                                                    
Hence, 
$$
   |\lambda| \left|\frac{\lambda}{\phi'(1)^k}\right|^n \le  \frac{ \beta^k\|g\|_{H^2(\U)}}{ |\langle z^m(1-z)^k, g\rangle| }
   $$ 
for every $n$,  a contradiction since $|\lambda/\phi'(1)^k| > 1$.  
\end{proof}

 The discussion of Section~\ref{ELFS} shows that if $\phi$ is an essentially linear fractional selfmap of $\U$ of parabolic nonautomorphism type, then $\phi$ satisfies the hypotheses of the following proposition. 
  
 \begin{prop} \label{PP1} Suppose that $\phi \in C^2(\dw)$ is such that $\phi'(\dw) =1$ and $\Re(\dw\phi''(\dw)) > 0$.  If $\lambda\notin\{0, 1\}$ is an eigenvalue of $C_\phi$, then $\lambda$ is not an isolated point of $\Sp(C_\phi)$.
 \end{prop}
 \begin{proof}    Suppose that $\lambda\notin\{0 , 1\}$ is an eigenvalue of $C_\phi$ so that $f\circ \phi = \lambda f$ for some $f\in H^2$, which is not the zero function.  Suppose that $f(z_0) = 0$ for some $z_0\in \U$.  Then because $f(\phi^{[n]}(z_0)) = \lambda^nf(z_0) =0$, we see that $f$ vanishes at every point in the orbit $\{\phi^{[n]}(z_0): n\ge 0\}$.  However,   by Lemma 4.5 of \cite{BoS},  $(\phi^{[n]}(z_0))$ is not a Blaschke sequence.  Thus we have an $H^2(\U)$ function $f$, which is not the zero function, vanishing on a  non-Blaschke sequence, which is  a contradiction. It follows that $f$ has no zeros in $\U$.    Hence there is an analytic function $g$ on $\U$ such that $f = e^g$ and  $f^{1/x}:=e^{g/x}$ is an $H^2(\U)$ function for each $x\ge 1$.   Let $\log \lambda$ represent some fixed choice of the logarithm of $\lambda$.    Since $ \lambda e^g =  f\circ \phi = e^{g\circ \phi}$, it follows that there is an integer $k$ such that $g\circ\phi = \log \lambda + g + 2\pi i k$.  Then $f^{1/x} \circ \phi = e^{1/x(\log\lambda + 2k\pi i)}f^{1/x}$ so that for each $x\ge 1$,   $\lambda^{1/x}:= e^{1/x(\log\lambda + 2ki\pi )}$ is an eigenvalue of $C_\phi$.  Thus  $\lambda$ belongs to the path $\lambda^{1/x}$, $x\ge1$ of eigenvalues, so that $\lambda$ is not an isolated point of $\Sp(C_\phi)$.
 \end{proof}
 
 The proof of the following lemma relies on some ideas from \cite[Section 3]{AB}.
 
 \begin{prop}\label{PP2}  Suppose that $\phi \in C^2(\dw)$ is such that $\phi'(\dw) =1$ and $\Re(\dw\phi''(\dw)) > 0$.  Then the point spectrum of $C_\phi$ has no interior.
 \end{prop}
 \begin{proof}  Let $k$ be an arbitrary positive integer and let $z_0\in \U$ also be arbitrary.  By Lemma 4.5 of \cite{BoS}, applied to $\phi^{[k]}$ (which also satisfies $\phi^{[k]}(\dw) = 1$ and $\Re(\dw(\phi^{[k]})''(\dw) >0$), the sequence  $(\phi^{[kn]}(z_0))$ is not Blaschke.  It follows that $(C_\phi^*)^k = C_{\phi^{[k]}}^*$ is a cyclic operator on $H^2$.  In fact, $K_{z_0}$ is a cyclic vector: if for some $f\in H^2$, we have $\langle f, (C_{\phi^{[k]}}^*)^nK_{z_0}\rangle = 0$ for all nonnegative integers $n$, then since $\langle f, (C_{\phi^{[k]}}^*)^nK_{z_0}\rangle = \langle f, K_{\phi^{[kn]}(z_0)}\rangle=f(\phi^{[kn]}(z_0))$, we see $f$ vanishes on a non-Blaschke sequence so that $f\equiv 0$.   
 
 We have shown that every positive integral power of $C_\phi^*$ is cyclic and it follows that no power of $C_\phi$ can have an eigenvalue of multiplicity greater than $1$.  In particular this means that if $\lambda$ is an eigenvalue of $C_\phi$ then $\zeta\lambda$ cannot be an eigenvalue of $C_\phi$ for $\zeta$ any root of unity different from $1$.  Thus, in particular, the point spectrum of $C_\phi$ has no interior.
 \end{proof}

\section{Spectra of Essentially Linear Fractional Composition Operators}

As we explained in Section~\ref{ELFS}, when characterizing spectra for composition operators induced by essentially linear fractional selfmaps $\phi$ of $\U$, the situation of interest is that where $\phi$ fixes a point on the unit circle.  Without loss of generality, we assume $\phi(1) = 1$.    We begin the the hyperbolic case. 

\begin{theorem}\label{HTELF} Suppose that $C_\phi$ is an essentially linear fractional selfmap of $\U$ of hyperbolic-type with Denjoy-Wolff point $1$.  Then
$$
\Sp(C_\phi) = \Sp_e(C_\phi) = \{z: |z| \le \phi'(1)^{-1/2}\}.
$$
\end{theorem}
\begin{proof}   Let $\phi'(1) = p$ so that $p< 1$ and let $a = \phi''(1)$.    Let $\Psi$ be the selfmap of the right halfplane $\Pi$ given by
$$
\Psi(w) = \frac{1}{p}w + \left(\frac{1}{p} - 1 + \frac{a}{p^2}\right).
$$
By the discussion of Section~\ref{ELFS}, we know that $C_\phi$ is compactly equivalent to $C_\psi$, where $\psi = T^{-1}\circ \Psi\circ T$.  The mapping $\psi$ is a linear-fractional  hyperbolic non-automorphism of $\U$ with Denjoy-Wolff point $1$ and $\psi'(1) = p$.  By, e.g., Corollary 4.8 of \cite{Cow2} (which characterizes the spectrum of any composition operator whose symbol is of hyperbolic type and analytic on the closed disc) the spectrum of $C_\psi$ is the disc $D:=\{z: |z| \le \psi'(\dw)^{-1/2}\}=\{z: |z| \le \phi'(\dw)^{-1/2}\}$.  Actually every point in $D\setminus \{0\}$  is an infinite multiplicity eigenvalue of $C_\psi$ so that $D= \Sp_e(C_\psi)$. Alternatively, $D = \Sp_e(C_\psi)$ because $\psi$ is horocyclic at $1$ (Proposition~\ref{HCP} above).   

Because $C_\phi$ and $C_\psi$ are equivalent modulo the compacts, we have $\Sp_e(C_\phi) = D$, but since $\phi$ is horocyclic at $1$, $\Sp(C_\phi) = \Sp_e(C_\phi)$, which completes the argument.
\end{proof}

Consider $\phi(z) = 2/(\sqrt{13-4z} -1)$, the hyperbolic-type essentially linear fractional selfmap of $\U$  discussed in Example~\ref{EIHT}.   By Theorem~\ref{HTELF}, the disc $\{z: |z| \le \sqrt{3}\}$ is both the spectrum and essential spectrum of $C_\phi$. We now turn to the case in which the essentially linear fractional selfmap has interior fixed point (as well as boundary fixed point $1$).  
  
\begin{theorem}\label{IFPELF}  Suppose that $\phi$ is an essentially linear fractional selfmap of $\U$ fixing $1$.  Suppose that $\phi'(1) > 1$ so that the Denjoy-Wolff point  $\dw$ of $\phi$ lies in $\U$.  Let $N$ be least positive integer for which $\phi'(\dw)^N \le \phi'(1)^{-1/2}$.  Then
$$
\Sp(C_\phi) = \{z: |z| \le \phi'(1)^{-1/2}\} \cup  \{\phi'(\dw)^n: n = 0,\ldots, N-1\}
$$
and 
$$
\Sp_e(C_\phi) =  \{z: |z| \le \phi'(1)^{-1/2}\}.
$$                                                                      
\end{theorem}

\begin{proof}   The spectrum of $C_\phi$ contains all powers of $\phi'(\dw)$ (\cite[Theorem 4.1]{Cow2}).  Moreover, work of Koenigs \cite{Kg} shows that the only possible eigenvalues of $C_\phi$ are powers of $\phi'(\dw)$.   

 Let $D =  \{z: |z| \le  \phi'(1)^{-1/2}\}$.  We prove that the disc $D$ equals $\Sp_e(C_\phi)$, and the proposition follows from the discussion of the preceding paragraph and that of Section~\ref{SESS} because any spectral point $\lambda$ outside of $D$ is in the unbounded component of the essential resolvent, making $\lambda$ an eigenvalue.

    Let $\psi$ be the linear-fractional self-mapping of $\U$ given by $T^{-1}\circ\Psi\circ T$, where 
$$
\Psi(w) = \frac{w}{p} + \left(\frac{1}{p}-1+ \frac{a}{p^2}\right),
$$
where $p = \phi'(1)>1$ and $a = \phi''(1)$.   From the discussion of Section~\ref{ELFS}, we know there is a constant $c > 0$ such that 
\begin{equation}\label{CIN}
\Re\left(\frac{1}{p}-1+\frac{a}{p^2}\right)\ge c,
\end{equation}
 and since $p>1$,  it follows that $\Re(a) > 0$ (just as in parabolic non-automorphism case). 
  We prove that the essential spectrum of $C_\psi$ is the disc $D$, which implies the same is true of $C_\phi$ since $C_\phi$ and $C_\psi$ differ by a compact operator.   

Applying either results of Kamowitz (see, e.g. p.\ 296 of \cite{CMB})  or Theorem 7.31 of \cite{CMB}, we see that the essential spectral radius of $C_\psi$ is $\psi'(1)^{-1/2} = \phi'(1)^{-1/2}$. Thus we need only show that every point in $D$ is in the essential spectrum of $C_\psi$.   

Applying the modified version (\ref{ORFM})  of Cowen's adjoint formula for linear fractional composition operators, we see
$$
C_\psi^* =  \Lambda_0 + M_{z\sigma'(z)} C_\sigma B,
$$
where $B$ is the backward shift and 
$$
 \sigma(z) = \frac{(-2p^2 +\bar{a})z - \bar{a}}{(2p-2p^2+\bar{a})z-2p-\bar{a}} \quad {\rm while}\quad z\sigma'(z) = \frac{4p^3z}{(2p-2p^2 + \bar{a})z - 2p - \bar{a})^2}.  
$$
Thus, since $\Lambda_0$ is a rank-one operator, we see that  $C_\psi^*$ is equivalent to the operator
$$
M_{z\sigma'(z)} C_\sigma B
$$
modulo the compacts.  

By Lemma 3.3 of \cite{BLNS},  The weighted composition operator   $M_{\sigma'(z)} C_\sigma$ is equivalent to $\sigma'(1)C_\sigma = \frac{1}{p}C_\sigma$ modulo the compacts, making
$C_\psi^*$ equivalent to $M_{z/p}C_\sigma B$ modulo the compacts.    We show that every point in the disc $D$ is an infinite multiplicity eigenvalue of $M_{z/p}C_\sigma B$.  Thus each point of $D$ is in the essential spectrum of $M_{z/p}C_\sigma B$, hence of $C_\psi^*$, hence of $C_\psi$, as desired.  

 The self-mapping $\sigma$ of $\U$ is of hyperbolic type with Denjoy-Wolff point $1$ and Denjoy-Wolff derivative equal to $\sigma'(1) = 1/p$.  Translated to the right-half plane $\Pi$ via $T$, the function $\sigma$ takes the form  $\Sigma(w) = pw + 1 - p + \bar{a}/p$ (i.e. $\Sigma = T\circ \sigma\circ T^{-1}$).  Note  $1-p+\bar{a}/p = p(1/p - 1 + \bar{a}/p^2)$ has positive real part  so that $\Sigma$ is indeed a self-mapping of $\Pi$.  It is easy to see that
 $$
 \Gamma(w) = w + \frac{p - p^2 + \bar{a}}{p(p-1)}
 $$
  is a self-mapping of $\Pi$ for which 
  \begin{equation}\label{RHPIR}
\Gamma\circ \Sigma =  p \Gamma.
  \end{equation}
  Note well that $(p-p^2+\bar{a})/(p(p-1))$ has strictly positive real part (the denominator is positive and $(p - p^2 + \bar{a}) = p^2(1/p - 1 +\bar{a}/p^2)$ has real part exceeding $p^2c$ by (\ref{CIN})), which means that  $\gamma(z) := (T^{-1}\circ \Gamma\circ T)(z)$ will take the unit disc $\U$ univalently onto a proper subdisc of $\U$ internally tangent to the unit disc at $1$.  In particular $\gamma(\U)$ will be bounded away from $-1$.  
   The intertwining relationship (\ref{RHPIR})  translates to the disc as follows
  $$
  \gamma \circ \sigma = \nu\circ \gamma
  $$
  where $\nu = T^{-1}\circ (p T)$ is a hyperbolic automorphism of $\U$ with $\dw = 1$ and $\nu\, '(1) = 1/p$.  Let $\ln$ be the principal branch of the logarithm function and recall that $f_b(z) = e^{b\ln((1+z)/(1-z))}$ belongs to $H^2(\U)$ whenever $-1/2 < \Re(b) < 1/2$.  Since $\gamma(\U)$ is bounded away from $-1$, we conclude that $f_b\circ \gamma$ belongs to $H^2(\U)$ whenever $-\infty < \Re(b) < 1/2$.   
Let $g(z) = zf_b(\gamma(z))$ so that $Bg = f_b\circ \gamma$. For every $b$ satisfying $-\infty < \Re(b) < 1/2$, we have
\begin{eqnarray*}
M_{z/p}(C_\sigma B g)(z) &=& M_{z/p}(f_b\circ\gamma \circ \sigma)(z)\\
  & = & \frac{z}{p} (f_b\circ \nu)(\gamma(z))\\
  &=& \frac{z}{p}(T\circ\nu)^b(\gamma(z))\\
  & = & \frac{z}{p} p^{b}f_b(\gamma(z))\\
  & = & p^{b -1} g(z).
  \end{eqnarray*}
  Fix $b$ with $-\infty < \Re(b) < 1/2$ and let $k$ be an integer.   Replacing $b$ with $b + 2\pi ki/\ln(p)$ in the preceding computation, we see that $$z\mapsto ze^{(b+2\pi ki/\ln(p))\ln((1+\gamma(z))/(1-\gamma(z)))}$$ is an eigenvector for $M_{z/p}C_\sigma B$ with corresponding eigenvalue $p^{b-1}$.  It follows that every point in the punctured disc $\{z: |z| < p^{-1/2}\}\setminus\{0\}$ is an infinite-multiplicity eigenvalue of the operator $M_{z/p} C_\sigma B$.  Hence the essential spectrum of  $M_{z/p} C_\sigma B$ contains the closure of this punctured disc, which is  $D$, as desired.
  \end{proof}

By the the preceding theorem, $\Sp(C_\phi) = \{z: |z| \le 1/\sqrt{2}\} \cup \{1\}$ and $\Sp_e(C_\phi) = \{z: |z| \le 1/\sqrt{2}\}$ when $\phi(z) = \frac{2z^2-3z+3}{2z^2-7z+7}$ is the essentially linear fractional selfmap of Example~\ref{ELFSQE} ($\phi(1/2) = 1/2$, $\phi'(1/2) = 3/8$, and $\phi'(1) = 2$).  

      As we have discussed, if $\phi$ is essentially linear fractional and of parabolic type with Denjoy-Wolff point $1$, then $\Re(\phi''(1)) > 0$.   We now complete our characterization of spectra and essential spectra for essentially linear fractional composition operators.
    
    \begin{theorem}\label{PTELF}  Suppose that $\phi$ is an essentially linear fractional selfmap of $\U$ fixing $1$.  Suppose that $\phi'(1) = 1$ so that the Denjoy-Wolff point  $\dw$ of $\phi$ is $1$.  Let $a =\phi''(1)$ so that $\Re(a) >0$.    Then
$$
\Sp(C_\phi)  = \Sp_e(C_\phi) =  \{e^{-at}:  t\ge 0\}\cup \{0\}.  
$$                                                                     
\end{theorem}
\begin{proof}  Let $\psi$ be the linear fractional selfmap of $\U$ given by $\psi = T^{-1}\circ\Psi\circ T$, where $\Psi$ is the selfmap of $\Pi$ given by
$$
\Psi(w) = w + a.
$$
By Theorem 6.1 of \cite{Cow2}, the spectrum of $C_\psi$ is is precisely the spiral $S :=  \{e^{-at}:  t\ge 0\}\cup \{0\}$.   Because $S$ has no interior and no isolated points, $S$ is the essential spectrum of $C_\psi$ as well.  Because $C_\phi$ and $C_\psi$ are equivalent modulo the compact operators, $\Sp_e(C_\phi) = S$, as desired.  

That $\Sp(C_\phi) = \Sp_e(C_\phi)$ follows quickly from Propositions~\ref{PP1} and \ref{PP2}.   Let $\lambda\in\Sp(C_\phi)\setminus \Sp_e(C_\phi)$ be arbitrary.  Note $\lambda$ is neither $0$ nor $1$.   Because the complement of $S$ is connected, $C_\phi - \lambda I$ must be Fredholm of index $0$.   Thus $\lambda$ must belong to the point spectrum of $C_\phi^*$ and from Proposition~\ref{PP1}, we see that $\lambda$ is not an isolated point of the spectrum of $C_\phi$.  Since $\lambda$ is in the essential resolvent and not isolated, it cannot be a boundary point of $\Sp(C_\phi)$ (see, e.g, \cite[Theorem 6.8, p.\ 366]{Con}).  This contradicts Proposition~\ref{PP2}: since $\Sp(C_\phi)\setminus \Sp_e(C_\phi)$ consists of eigenvalues,  every point in $\Sp(C_\phi)\setminus \Sp_e(C_\phi)$  must be a boundary point of $\Sp(C_\phi)$.    We conclude $\Sp(C_\phi)\setminus \Sp_e(C_\phi)$ is empty, as desired.
\end{proof}

Applying the preceding theorem, we see that the spiral $\{e^{-2t - it}: t\ge 0\} \cup\{0\}$ is the spectrum and essential spectrum of the parabolic-type selfmap $\phi$ of Example~\ref{EIPC}.


\section{Spectra of Some Weighted Composition Operators}
In this section, we consider the spectrum of weighted composition operators on $H^2(\U)$ of the form  $C_{g, \phi}$, where $\phi$ is an analytic selfmap of $\U$ and $g\in H^\infty(\U)$.  Here, we have $C_{g, \phi} f = g f\circ \phi$  for $f\in H^2(\U)$.  Clearly,  $C_{g, \phi}$ is a bounded linear operator on $H^2(\U)$ and $\|C_{g,\phi}:H^2(\U)\rightarrow H^2(\U)\| \le \|g\|_\infty\|C_\phi: H^2(\U)\rightarrow H^2(\U)\|$.  

We begin with a couple results that do not require $\phi$ to be essentially linear fractional.

\begin{lemma}\label{ELWC}  Suppose that $\phi$ is an {\em arbitrary}  analytic selfmap of $\U$ with Denjoy-Wolff point $\dw$.   Suppose that $g\in H^\infty(\U)$  extends to be continuous on $\U\cup \{\dw\}$ (if $\dw \in \partial \U$).  If $\lambda$ is an eigenvalue of $C_{g, \phi}$, then  $|\lambda| \le |g(\dw)|\, r(C_\phi)$, where $r(C_\phi)$ is the spectral radius of $\phi$.  If $g(\dw) = 0$ and $\phi$ and $g$ are nonconstant, then $C_{g,\phi}$ has no eigenvalues.  
\end{lemma}
\begin{proof}    Suppose that $\lambda$ is an eigenvalue for $C_{g, \phi}$   with corresponding eigenvector $f$.   For each positive integer $n$ and $z\in \U$, we have
\begin{equation}\label{HCEA}
\lambda^n f(z)= \prod_{j=0}^{n-1} g(\phi^{[j]}(z)) f(\phi^{[n]}(z)).
\end{equation}
Observe that for any fixed $z\in \U$ and positive integer $n$, we have from (\ref{PWB}):
\begin{equation}\label{SOEA}
|f(\phi^{[n]}(z))| \le \frac{\|f\circ\phi^{[n]}\|_{H^2(\U)}}{\sqrt{1-|z|^2}}\le \|C_{\phi}^n\| \frac{ \|f\|_{H^2(\U)}}{\sqrt{1-|z|^2}}.
\end{equation}

Choose $z\in \U$ such that $f(z)\ne 0$.  Since $(\phi^{[j]}(z))$ approaches $\dw$ as $j\rightarrow \infty$, we know $g(\phi^{[j]}(z))$ approaches $g(\dw) = 0$ as $j\rightarrow\infty$.  Upon (i) taking $n$-th roots of the absolute value each side of (\ref{HCEA}), (ii) using the estimate (\ref{SOEA}),  and (iii) letting $n\rightarrow \infty$, we obtain
\begin{equation}\label{ERFS}
 |\lambda|\le |g(\dw)|\, r(C_\phi),
 \end{equation}
as desired.   

Now assume $g(\dw)  = 0$ while $g$ and $\phi$ are nonconstant.  Then by (\ref{ERFS}), $\lambda = 0$ is the only possible eigenvalue for $C_{g,\phi}$.  However,   $g$, being nonconstant,  is not the zero function; since $\phi$ is also nonconstant,  $0$ cannot be an eigenvalue of $C_{g,\phi}$.  
\end{proof}

We can weaken the hypothesis on $g$ by making further assumptions on $\phi$.

\begin{lemma}\label{ELWC2} Suppose that  (i) $\phi$ is an arbitrary selfmapping of $\U$ of hyperbolic type or  (ii) $\phi \in C^2(1)$ satisfies  $\phi(1) = 1$, $\phi'(1) = 1$, and $\Re(\phi''(1)) > 0$.  If $\lambda$ is an eigenvalue of $C_{g,\phi}$ and  $g\in H^\infty(\U)$ has finite radial (equivalently nontangential) limit $g(\dw)$  at $\dw$, then
$$
|\lambda| \le |g(\dw))|\, r(C_\phi);
$$
moreover, $C_{g,\phi}$ will have no eigenvalues if $g(\dw) = 0$, and $\phi$ and $g$ are nonconstant functions.  

\end{lemma}
\begin{proof}  In cases (i) and (ii) for each $z\in \U$, the sequence $(\phi^{[j]}(z))$ converges nontangentially to $\dw$ (for case (i) see \cite[Lemma 2.66]{CMB}; for case (ii), see Lemma 4.5 of \cite{BoS}).    Thus  $(g(\phi^{[j]}(z))$ will converge to $g(\dw)$ as $j\rightarrow \infty$ and the argument of Lemma~\ref{ELWC} may be applied to complete the proof.
\end{proof}

We now turn our attention to weighted composition operators $C_{g, \phi}$, where $\phi$ is essentially linear fractional.  
 Suppose that $\phi$ is essentially linear fractional, so that there are unimodular constants $\eta$ and $\zeta$ such that  $\phi\in C^3(\zeta)$, $\phi(\zeta) = \eta$, $\phi^{-1}(\{\eta\}) = \{\zeta\}$,  and $\phi(\U)$ lies in a proper subdisc of $\U$ internally tangent to the unit circle at $\eta$.  Just as in Section~\ref{ELFS}, we see that if $\eta\ne\zeta$, then $\|\phi\circ \phi\|_\infty < 1$ and $(C_{\phi})^2$ is thus compact.  It follows that $(C_{g, \phi})^2 = g\, g\circ\phi\,  (C_\phi)^2$ is also compact and its spectrum is $\{g(\dw)\phi'(\dw)^n: n = 0, 1, 2, \ldots\}\cup \{0\}$ by, e.g, \cite[Corollary 1]{GG}  or \cite[p.\ 187]{CFD}.   Thus, just as in the preceding section, we focus on the situation where $\phi$ fixes a point on $\partial \U$ and we assume, without loss of generality, that this point is $1$.

 We require the following lemma, which is a generalization of Lemma 3.3 of \cite{BLNS}.  The structure of our proof is the same as the one in \cite{BLNS}.
 
 \begin{lemma}\label{L33}  Suppose that $\phi$ is an essentially linear fractional selfmapping of $\U$  such that $\phi(1) = 1$, so that $\phi$ is horocyclic at $1$ and $\phi^{-1}(\{1\}) = \{1\}$.  In addition, suppose that for some $\delta > 0$, $\phi'$  is continuous on $ \{z: |z-1| < \delta \ \text{and}\ |z| \le 1\}$.    Suppose that $g\in H^\infty(\U)$ has radial limit function $g$ that is defined on some open arc $A$ of the unit circle including $1$, that $g(1) = 0$,  and that $g$ restricted to  $A$ is differentiable at $1$.   Then $C_{g,\phi}$ is a compact operator on $H^2(\U)$.
 \end{lemma}
 \begin{proof}    Our hypotheses on $g$ show that there is a constant $C_1$ such that $|g(e^{it})| < C_1|t|$ for almost every $t\in [-\pi, \pi)$ (with respect to Lebesgue measure).  Similarly, our hypotheses on $\phi$ show that there is a constant $C_2>0$ such that  $|1- \phi(e^{it})| \ge C_2|t|$ for almost every $t\in [-\pi, \pi)$.   To obtain the latter estimate, note that $\phi'(1)$ must be positive since $\phi$ is a selfmap of $\U$. Thus our differentiability hypothesis for $\phi$ on $\U^-\cap \{z: |z-1| < \delta\}$  shows that there is an open interval $E$ containing $0$ and a constant $c_1$ such that $|1-\phi(e^{it})| \ge c_1|t|$ for each $t \in E$.  In addition, there must be a constant $c_2 > 0$ such that  $|1- \phi(e^{it})| \ge c_2$ for almost every $t\in [-\pi, \pi)\setminus E$ for otherwise $1$ would be a cluster point of some $\zeta\in \partial \U \setminus \{e^{it}: t\in E\}$, contradicting $\phi^{-1}(\{1\}) = \{1\}$.  Thus there must be a constant  $C_2>0$ such that  $|1- \phi(e^{it})| \ge C_2|t|$ for almost every $t\in [-\pi, \pi)$.
 
  We show that $C_{g, \phi}: H^2(\U)\rightarrow H^2(\U)$ is actually a Hilbert-Schmidt operator.  Because $(z^n: n =0, 1, \ldots)$ is an orthonormal basis of $H^2(\U)$, it suffices to show that
 $$
 \sum_{n=0}^\infty \|C_{g, \phi} z^n\|^2 < \infty.
 $$
 We have for each nonnegative integer $n$.
 \begin{eqnarray*}
\|C_{g, \phi} z^n\|^2 =  \frac{1}{2\pi} \int_{-\pi}^\pi  |g(e^{it})|^2 |\phi(e^{it})|^{2n}\, dt \\
                                    \le\frac{C_1^2}{2\pi}  \int_{-\pi}^\pi  |t|^2 |\phi(e^{it})|^{2n}\, dt.
                                   \end{eqnarray*}
Summing both sides of the inequality just established, interchanging the sum and the integral, and applying the the fact that $|\phi(e^{it})| < 1$ a.e. (since $\phi$ is horocyclic at $1$), we obtain
\begin{equation}\label{LE}
\sum_{n=0}^\infty \|C_{g, \phi} z^n\|^2 \le \frac{C_1^2}{2\pi}  \int_{-\pi}^\pi  \frac{|t|^2}{1- |\phi(e^{it})|^2}\, dt.
\end{equation}
Because $\phi(\U)$ lies in a horodisc of the form $H(1, \beta)  = \{z: |1- z|^2 < \beta(1- |z|^2)\}$ for some $\beta >0$, it follows that $|1-\phi(e^{it})|^2 \le \beta(1-|\phi(e^{it})|^2)$ for almost every $t\in [-\pi, \pi)$.   Combining this estimate with the one discussed in the initial paragraph of the proof  ($|1- \phi(e^{it})| \ge C_2|t|$), we see that for almost every $t\in [-\pi, \pi)$,
$$
1-|\phi(e^{it})|^2 \ge \frac{C_2^2}{\beta}|t|^2.
$$
Thus the integral on the right of equation (\ref{LE}) is finite and $C_{g, \phi}$ is Hilbert-Schmidt, as claimed. 
\end{proof}

  The following theorem generalizes Theorem 3 of \cite{GG}.  
  \begin{theorem}\label{GGG} Suppose that $\phi$ and $g$ are as in Lemma~\ref{L33}.  Then $\Sp(C_{g,\phi}) = \{0\}$.
  \end{theorem}
  \begin{proof}  By Lemma~\ref{L33}, $C_{g,\phi}$ is compact and thus $\Sp_e(C_{g,\phi}) = \{0\}$.  Hence any nonzero spectral point must be an eigenvalue of $C_{g,\phi}$.  If $\phi$'s Denjoy-Wolff point $\dw$ belongs to $\U$, then apply Lemma~\ref{ELWC} to see that $C_{g,\phi}$ has no nonzero eigenvalues. If $\dw = 1$, then $\phi$ is as in either case (i) or case (ii) of Lemma~\ref{ELWC2}  so that, again, $C_{g,\phi}$ has no nonzero eigenvalues.
  \end{proof}
  
 We now characterize the spectrum and essential spectrum of $C_{g,\phi}$ when $\phi$ and $g$ satisfy the hypotheses of Lemma~\ref{L33}  and either $\phi'(1) <1$ or $\phi'(1) > 0$.   Our final result will be a spectral  characterization  for $C_{g, \phi}$ in parabolic case $\phi'(1) =1$, but this result significant additional assumptions on $\phi$.  
 
  \begin{theorem}\label{HCWCT}   Suppose that $\phi$ and $g$ are as in Lemma~\ref{L33} and $\phi'(1) < 1$.   Then
  $$
  \Sp(C_{g,\phi}) = \Sp_e(C_{g,\phi}) = \{z: |z| \le |g(1)|\phi'(1)^{-1/2}\}.
  $$
\end{theorem}  
\begin{proof}    By Lemma~\ref{L33}, $C_{g-g(1),\phi}$ is compact and thus $C_{g,\phi}$ is equivalent to $C_{g(1),\phi}$ modulo the compact operators.  By Theorem~\ref{HTELF}, the essential spectrum of $C_{g(1),\phi}$  is    $D:=\{z: |z| \le |g(1)|\phi'(1)^{-1/2}\}$ (because $C_{g(1),\phi} =g(1)C_\phi$)  and thus,
$ \Sp_e(C_{g,\phi}) =  D$ as well.   

Suppose that $C_{g,\phi}$ has a spectral point $\lambda$ outside $D$, then $\lambda$  must be an eigenvalue of $C_{g,\phi}$. However, Lemma~\ref{ELWC2} ensures that all eigenvalues of $C_{g,\phi}$ must belong to $D$ (because $r(C_\phi) = \phi'(1)^{-1/2}$ by \cite[Theorem 2.1]{Cow2}) and we conclude the spectrum of $C_{g,\phi} = D$, as desired.
\end{proof}

Let $\phi(z) = 2/(\sqrt{13-4z}-1)$  and $g(z) = 1/\sqrt{3} + (1-z)^{3/2}$.  Then, by Theorem~\ref{HCWCT}, the closed unit disc is both the spectrum and essential spectrum of $C_{g, \phi}$.  

 \begin{theorem}\label{IFPWCT}  Suppose that  $\phi$  and $g$ are as in Lemma~\ref{L33} and  $\phi'(1)> 1$ so that the Denjoy-Wolff point $\dw$ of $\phi$ belongs to $\U$.  Let $N$ be the least nonnegative integer such that $|g(\dw)||\phi'(\dw)|^N \le |g(1)||\phi'(1)|^{-1/2}$.  Then
  $$
 \Sp(C_{g,\phi}) = \{z: |z| \le |g(1)|\phi'(1)^{-1/2}\} \cup \{g(w)\phi'(\dw)^n: n = 0, 1, \ldots, N-1\}
  $$
  and
  $$
 \Sp_e(C_{g,\phi}) = \{z: |z| \le |g(1)|\phi'(1)^{-1/2}\}
  $$
\end{theorem}  
 \begin{proof}   The proof is similar to that of Theorem~\ref{HCWCT}.  Just as before, $C_{g,\phi}$ is equivalent to $C_{g(1),\phi}$ modulo the compact operators.  By Theorem~\ref{IFPELF}, the essential spectrum of $C_{g(1),\phi}$ is    $D:=\{z: |z| \le |g(1)|\phi'(1)^{-1/2}\}$ and thus,
$ \Sp_e(C_{g,\phi}) =  D$ as well.   

Any spectral point   of  $C_{g,\phi}$  outside $D$ must be an eigenvalue of $C_{g,\phi}$.  However, it's not difficult to see that eigenvalues of $C_{g,\phi}$ must have the form $g(\dw)\phi'(\dw)^n$ for some nonnegative integer $n$ (see, e.g, \cite[Proof  of Lemma 1]{GG}) and that all points of this form are in the spectrum of $C_{g,\phi}$ (\cite[Lemma 3]{GG}).  The theorem follows.
\end{proof}

For a concrete example illustrating the preceding theorem, set $g(z) = (3-2z)^3$ and $\phi(z)=  \frac{2z^2-3z+3}{2z^2-7z+7}$ .  Then $\phi$ is the essentially linear fractional map of Example~\ref{ELFSQE}, which satisfies $\phi(1) = 1, \phi'(1) = 2, \phi(1/2) = 1/2$, and $\phi'(1/2) = 3/8$. Here $g(1) =1$ while $g(1/2) = 8$.   Thus $|g(1/2)||\phi'(1/2)|^n > |g(1)| |\phi'(1)|^{-1/2}$ for precisely $n = 0, 1$, and  $2$. Hence, 
$$
\Sp(C_{g,\phi}) = \{z: |z| < 1/\sqrt{2}\} \cup \{8, 3, 9/8\}
$$
 and $\Sp_e(C_{g,\phi}) = \{z: |z| < 1/\sqrt{2}\}$.

Only with significantly stronger hypotheses on $\phi$ can we obtain a spectral characterization for $C_{g, \phi}$  when $\phi$ is of parabolic type.   We suppose that $\phi$ satisfies the hypotheses of Lemma~\ref{L33} and in addition that $\phi$ is continuous on the closed disc $\U^-$, one-to-one on $\U^-$, and is $C^{3+\epsilon}(1)$ for some $\epsilon > 0$. Moreover, we assume that 
 $$
 \Re(\overline{\phi''(1)}(\mathcal{S}\phi)(1)) \ge 0.
 $$
 where
 $$\mathcal S\phi(1) = \left(\frac{\phi''}{\phi'}\right)'(1) - \frac{1}{2}\left(\frac{\phi''}{\phi'}\right)^2(1)$$
 is the Schwarzian derivative of $\phi$ at $1$.    With all these hypotheses on $\phi$,  the ``Parabolic Models'' Theorem 4.12 of \cite{BoS} may be applied to $\phi$, or, more precisely to the right halfplane incarnation $\Phi$ of $\phi$ (so that $\Phi = T\circ\phi \circ T^{-1}$).  The Parabolic Models Theorem from \cite{BoS} tells us that there is an analytic mapping $\nu$ defined on the right halfplane $\Pi$ such that $\nu\circ\Phi = \nu + \phi''(1)$.  Moreover by part (c) of Theorem 4.12 (see also the discussion of Schwarzian derivatives on pages 51 and 52 of \cite{BoS}), we can assume that $\nu$ is a selfmap of the right halfplane $\Pi$.  Thus  $\nu\circ T\circ \phi = \nu\circ T + \phi''(1)$ and since $\nu$ is a selfmap of  $\Pi$, for each $t \ge 0$, we see that $h(z) = e^{-t\nu\circ T}$ belongs to $H^\infty(\U)$ and
\begin{equation}\label{ERPC}
C_\phi h = e^{-t\phi''(1)} h.
\end{equation}
 
\begin{theorem}\label{PCWCT}  Suppose that $\phi$ and $g$ are as in Lemma~\ref{L33}; in addition, suppose that  $\phi$ is continuous on the closed disc $\U^-$, one-to-one on $\U^-$, and is $C^{3+\epsilon}(1)$ for some $\epsilon > 0$. Moreover, assume that 
 $$
 \Re(\overline{\phi''(1)}(\mathcal{S}\phi)(1)) \ge 0.
 $$
 where
 $\mathcal (S\phi)(1)$ is the Schwarzian derivative of $\phi$ at $1$.   Then
  $$
 \Sp(C_{g,\phi}) = \Sp_e(C_{g,\phi})=  \{g(1)e^{-bt}: t\ge 0\} \cup \{0\},
  $$
where $b = \phi''(1)$.  
\end{theorem}

\begin{proof}  If $g(1) = 0$, the theorem follows from Theorem~\ref{GGG} above.   Thus, we assume $g(1)\ne 0$.   

Just as in the proofs of Theorems \ref{HCWCT} and \ref{IFPWCT} above,  $C_{g,\phi}$ is equivalent to $C_{g(1),\phi}$ modulo the compact operators.  The essential spectrum of $C_{g(1),\phi}$  is $S:= \{g(1)e^{-bt}: t\ge 0\} \cup \{0\}$ (by Theorem~\ref{PTELF} since $C_{g(1),\phi} = g(1)C_\phi$).   Thus,
$ \Sp_e(C_{g,\phi}) = S$.

  The argument that there are no points in the spectrum of $C_{g,\phi}$ outside of $S$ is similar to that of Theorem~\ref{PTELF}.    

     As we have discussed (see Section~(\ref{ELFS})), since $\phi$ is essentially linear fractional and of parabolic type, $\Re(\phi''(1)) > 0$.    Hence,  by Lemma 4.5 of \cite{BoS}, applied to $\phi^{[k]}$ the iterate sequence $(\phi^{[kn]}(z))$ fails to be Blaschke for every $z\in \U$.   We claim that every power of $C_{g, \phi}^*$ is cyclic on $H^2(\U)$.   Since $g$  has radial/nontangential limit $g(1)$ at $1$ and we are assuming $g(1) \ne 0$ and since $(\phi^{[j]}(0))$ approaches $1$ nontangentially (\cite[Lemma 4.5]{BoS}),  there is a positive integer $N$ such that for all $j\ge N$, $g(\phi^{[j]}(0))$ is nonzero.  Set $z_0 = \phi^{[N]}(0)$.   Then $K_{z_0}$ will be cyclic for the $k$-th power of $C_{g,\phi}^*$:  $((C_{g,\phi}^*)^{kn}K_{z_0} =  (C_\phi^*(M_g)^*)^{kn} = \prod_{j=0}^{kn-1} \overline{g(\phi^{[j]}(z_0))} K_{\phi^{[kn]}(z_0)}$ so that if $f\in H^2(\U)$ is orthogonal to all powers of $(C_{g,\phi})^*$, it must vanish on the non-Blaschke sequence $(\phi^{[kn]}(z_0))$ and hence $f\equiv 0$.   Thus, we have established our claim: every power of $C_{g,\phi}$ is cyclic.  
     
      If follows, just as in the proof of Proposition~\ref{PP2}, that every spectral point of $C_{g,\phi}$ outside of the essential spectrum must be be boundary point of the spectrum and thus must be an isolated point of the spectrum.  We show no spectral point of $C_{g,\phi}$ outside of $S$ can be isolated, completing the proof of the theorem.

  Suppose that $\lambda\in \Sp(C_{g,\phi})\setminus S$ so that $\lambda$ is an eigenvalue of $C_{g,\sigma}$. Note $\lambda\ne 0$.  Let $f\in H^2(\U)$ be an eigenfunction associated with $\lambda$.   Since
 $\Re(\overline{\phi''(1)}(\mathcal{S}\phi)(1)) \ge 0$, there is an analytic self map $\nu$ of the right halfplane such that $\nu\circ T \circ \phi = \nu\circ T + \phi''(1)$.  Let $t\ge 0$, let $h(z) = e^{-t\nu\circ T}$.  Since $h\in H^\infty(\U)$, $hf\in H^2(\U)$; moreover, applying (\ref{ERPC}) and $gf\circ \phi = \lambda f$, we see  for each $t\ge 0$,
  $$
 C_{g,\phi} hf =   (h\circ \phi)(g f\circ\phi)= e^{-t\phi''(1)} \lambda hf. 
 $$
Thus for each $t\ge 0$, we see $e^{-t\sigma''(1)}\lambda$ is an eigenvalue of $C_{g, \phi}$ so $\lambda$ is not an isolated spectral point, which completes the proof of the theorem. 
\end{proof}

For a concrete example illustrating the preceding theorem consider the essentially linear fractional mapping $\phi$ of Example~\ref{EIPC}.   Recall $\phi = T^{-1}\circ \Phi\circ T$ where $\Phi(w) = w + 2+i - 1/(4(w+1)) - 1/(4(w+1)^{3/2}$.  It is easy to see that $\Phi'$ has positive real part on a neighborhood of the closed right halfplane and thus $\phi$ is  one-to-one on the closed disc.  Also $\phi\in C^{3+\epsilon}(1)$ for, say, $\epsilon = 1/4$.  Finally,  using $\phi''(1) = 2+i$ and $\phi'''(1) = 39/8 + 6i$, one calculates the Schwarzian derivative of $\phi$ at $1$ to be $3/8$.  Thus $\Re(\overline{\phi''(1)}(\mathcal{S}\phi)(1)) = 3/4 \ge 0$.  Letting $g(z) = i\sqrt{5 - z}$, e.g, we see $\Sp(C_{g,\phi}) = \Sp_e(C_{g,\phi}) = \{2ie^{-2t-it}: t\ge 0\}\cup \{0\}$.

\section{Open Questions}

Aside from questions Q1  and Q2 raised in the Introduction, the following questions seem interesting.
\begin{enumerate}
\item  What is the  compression spectrum of $C_\phi:H^2(\U)\rightarrow H^2(\U)$? 
\item  If $\phi$ is of hyperbolic type or parabolic automorphism type, is the compression spectrum of $C_\phi$ at most $\{0\}$?  
\item To what extent can the spectral characterizations of Theorems  \ref{HTELF}, \ref{IFPELF}, and \ref{PTELF} for essentially linear fractional composition operators be generalized? For example,  can analogs  be shown to hold under only the assumption that $\phi$ is horocyclic at a boundary fixed point?
\item   For a weighted composition operator $C_{g, \phi}$, where $g$ and $\phi$ satisfy the hypotheses of either Theorem~\ref{HCWCT} or Theorem~\ref{PCWCT},  the spectrum and essential spectrum of $C_{g, \phi}$ are the same.  Can $\Sp(C_{g,\phi}) = \Sp_e(C_{g,\phi})$  be established for a wider variety of combinations of $g$ and $\phi$?
\end{enumerate}
As we mentioned in Section~\ref{GRCOS},  if the answer to the second question above is ``yes'', then Q1 from the Introduction has an affirmative answer when $\phi$ is of hyperbolic or parabolic-automorphism type. We remark that  for  maps $\phi$ of parabolic automorphism type (but not automorphisms),  Cowen \cite[Section 7]{Cow2} conjectures that $\Sp(C_\phi) = \Sp_e(C_\phi) = \U^-$.


\end{document}